\documentclass[11p,leqno]{amsart}
\usepackage[english]{babel}
\usepackage{amsmath,amsthm}
\usepackage{amsfonts}
\usepackage{amssymb}
\usepackage{graphicx}
\usepackage{color}
\usepackage{amsfonts,mathtools,mathrsfs}
\usepackage{amscd}

\newtheorem{theorem}{Theorem}[section]
\newtheorem*{theorem*}{Theorem}

\numberwithin{equation}{section}

\begin{document}
\title[Liouville theorem for steady-state solutions of ELs]{Liouville theorem for steady-state solutions of simplified Ericksen-Leslie system}

\author{Yihang Hao}
\address{Institute of Mathematics, Hangzhou Dianzi University, Hangzhou, China}
\email{hanghy@hdu.edu.cn}
\author{Xian-gao Liu}
\address{School of Mathematic Sciences, Fudan University, Shanghai, China}
\email{xgliu@fudan.edu.cn}
\author{Xiaotao Zhang*}
\address{South China Research Center for Applied Mathematics and Interdisciplinary Studies, South China Normal University, Guangzhou, China}
\email{xtzhang@m.scnu.edu.cn}
\thanks{* Corresponding author.}
\maketitle

\begin{abstract}
In this paper, we investigate the steady simplified Ericksen-Leslie system. For three dimension, we obtain Liouville theorems if $u$ and $\nabla d$ satisfy the Galdi's\cite{MR1284206} condition, or some decay conditions. Note that the simplified Ericksen-Leslie system becomes Navier-Stokes equations when $d$ is constant vector.
\end{abstract}
\maketitle

\section{Introduction}

We consider the steady-state solutions of simplified Ericksen-Leslie system
\begin{equation}\label{A1}\left
\{\begin{array}{l}
\large{-\Delta u+u\cdot\nabla u +\nabla P=-\text{div}(\nabla d \odot \nabla d),}\\
\large{\text{div}\ u=0,}\\
\large{-\Delta d +u\cdot \nabla d=|\nabla d|^2 d},\\
\end{array}
\right.
\end{equation}
in $\mathbb{R}^3$, where $u:\mathbb{R}^3 \rightarrow \mathbb{R}^3$ is the velocity field, $P:\mathbb{R}^3\rightarrow \mathbb{R}$ is the scalar pressure and $d:\mathbb{R}^3\rightarrow \mathbb{S}^{2}$ is a unit vector field representing the macroscopic orientation of the nematic liquid crystal molecules. Here $\nabla d \odot \nabla d$ is a symmetric tensor with its component  $(\nabla d \odot \nabla d)_{ij}$
given by $\nabla_i d \cdot \nabla_j d=\sum_{k=1}^n\frac{\partial d_k}{\partial x_i} \cdot \frac{\partial d_k}{\partial x_j}$.
And, along with (\ref{A1}), the additional condition at infinity is as follows:
$$|u(x)|+|\nabla d(x)| \rightarrow 0, \quad \text{as  } |x| \rightarrow \infty.$$

The general Ericksen-Leslie system is modeling the hydrodynamic flow of nematic liquid crystal materials, it is a macroscopic continuum description of the evolution of the material under the influence of fluid velocity field and the macroscopic description of the microscopic orientation of fluid velocity $d$ of rodlike liquid crystals(see \cite{Leslie1968, MR0137403,de1974}). The simplified Ericksen-Leslie system was first proposed by Lin \cite{MR1003435}, it has attracted a lot of interest and generated a lot of interesting research work recently. For two dimensional space, the existence of global Leray-Hopf type weak solutions to the initial and boundary value problem has been proved by Lin-Lin-Wang\cite{MR2646822} and Lin-Wang \cite{MR2745211}; see also Hong \cite{MR2745194} and Xu-Zhang \cite{MR2853534} for related works. For three dimensional space, Lin-Wang \cite{MR3518239} proved the existence of global Leray-Hopf type weak solutions with the condition of the initial and boundary value satisfying $d_0\in \mathbb{S}^2_+$, i.e., $d_0$ takes values in the upper hemisphere.

When $d$ is constant vector, system (\ref{A1}) becomes the well-known stationary Navier-Stokes equations:	
\begin{equation}\label{ns}\left
	\{\begin{array}{l}
	\large{-\Delta v+v\cdot\nabla v +\nabla \pi=0,}\\
	\large{\text{div}\ v=0,}\\
	\end{array}
	\right.
\end{equation}
with the additional condition at infinity:
	\begin{equation}\label{nsb}
	|v(x)|\rightarrow 0\ \ as\ \ |x|\rightarrow 0,
    \end{equation}
where $v : R^n\rightarrow R^n$ is the velocity field, $\pi:  R^n\rightarrow R^n$ is the scalar pressure.
A long-standing open problem is whether the weak solutions of  (\ref{ns}), (\ref{nsb}) with
\begin{equation*}
  \int_{R^3}|\nabla v|^2dx<+\infty\nonumber
\end{equation*}
is trivial(namely, $v=0$, on $\mathbb{R}^3$).
Since then, many applauding results to the open problem have been established. For 3-D case,
Galdi\cite{MR1284206} first proved Liouville theorem for steady-state Navier-Stokes equations with the
condition $v\in L^\frac{9}{2}(\mathbb{R}^3)$, which we call $u$ satisfies Galdi's condition. Recently, Chae-Wolf\cite{MR3548261} improved Galdi's results with the condition $\int_{\mathbb{R}^3} |u|^{\frac92}\{\ln(2+\frac{1}{|u|})\}^{-1}dx <\infty$; Seregin\cite{MR3538409} proved that the velocity field belonging globally to $v\in L^6(\mathbb{R}^3)\cap BMO^{-1}(\mathbb{R}^3)$ is trivial; Kozono-Terasawa-Wakasugi\cite{MR3571910} proved $u=0$ with the condition vorticity $w=o(|x|^{-\frac53})$ or $||u||_{L^{\frac92,\infty}(\mathbb{R}^3)}\leq\delta (\int_{\mathbb{R}^3}|\nabla u|^2)^{1/3} $ for a small constant $\delta$.  Many other results one can refer Chae\cite{MR3162482}, Chae-Yoneda\cite{MR3061045}, Carrillo-Pan-Zhang\cite{CPZ}, Seregin\cite{SG} and Seregin-Wang\cite{SW} and the references therein. For 2-D and axially-symmetric cases, one can refer \cite{MR3014091, MR3543547, MR3345358} and the references therein for more details. Recently, Wendong Wang\cite{dong} and Na Zhao\cite{na} obtained the Liouville Theorem with the conditions that $v$ is axially symmetric and $|v(x_1,x_2,x_3)|\leq C/(1+r\prime)^\alpha$, where $\alpha>\frac23$ and $r\prime=\sqrt{x_1^2 +x_2^2}$.

When $u\equiv 0$, system (\ref{A1}) becomes the harmonic maps:
\begin{equation}\label{harmonic maps00}
  -\Delta d =|\nabla d|^2 d.
\end{equation}
Yau\cite{MR0431040} proved Liouville theorem under the hypothesis that the domain has nonnegative Ricci curvature. Under the same assumption, a generalization to harmonic maps into upper hemisphere and hyperbolic space is contained in the works of Cheng\cite{MR573431} and Choi\cite{MR647905}, see Tam\cite{MR1362965} and Shen\cite{MR1333944} for other cases. Schoen and Uhlenbeck\cite{MR762354} established regularity results and Liouville theorems for minimizing harmonic maps into the Euclidean sphere. Xin\cite{MR895408} generalized these results. Jin\cite{MR1156381} proved Liouville theorem under assumptions on the asymptotic behavior of the maps at infinity. Rigoli-Setti\cite{MR1695783} considered the rotationally symmetric harmonic maps into $\mathbb{R}^n$, upper hemisphere and hyperbolic space. Sinaei\cite{MR3290380} proved Liouville theorems for subharmonic functions.

Before stating our main results, let us introduce some notations. Throughout this paper, we use $L^r=L^r(\mathbb{R}^3)$ denote the standard Lebesgue spaces in $\mathbb{R}^3$, where $r\in [1,\infty]$. We denote
$$D_0^1=\{u\in L^6\mid ||\nabla u||_{L^2} <\infty\},$$
$$\mathring{H}=\{ u\in L^1_{loc}(\mathbb{R}^3) \mid ||\nabla u||_{L^2} < \infty\}.$$

Our main results are the following:
\begin{theorem}\label{u=0}
Assume that $u\in D_0^1$, $d\in L^\infty \cap \mathring{H}$ be solutions to (\ref{A1}) satisfying
\begin{equation}\label{Integration-condition}
 |u|+|\nabla d| \in L^{\frac92}.
\end{equation}

Then, we obtain that $u=0$ and $d$ is constant vector.
\end{theorem}

\begin{theorem}\label{u=0,decay}
Assume that $u\in D_0^1$, $d\in L^\infty \cap \mathring{H}$ be solutions to (\ref{A1}) satisfying
\begin{equation}\label{decay}
 |u(x_1,x_2,x_3)|+|\nabla d(x_1,x_2,x_3)| \leq \frac{C}{|1+\sqrt{x_1^2+x_2^2}|^\alpha},
\end{equation}

where $\alpha >\frac23$. Then, we obtain that $u=0$ and $d$ is constant vector.
\end{theorem}
We note that our result doesn't need the solution be axially symmetric, that is different from \cite{dong} and \cite{na}.
And the following result just need some decay along one direction.
\begin{theorem}\label{u=0,decay-1}
Assume that $u\in D_0^1$, $d\in L^\infty \cap \mathring{H}$ be solutions to (\ref{A1}) satisfying
\begin{equation}\label{decay-1}
 |u(x_1,x_2,x_3)|+|\nabla d(x_1,x_2,x_3)| \leq \frac{C}{|1+|x_3||^\beta},
\end{equation}

where $\beta >1$. Then, we obtain that $u=0$ and $d$ is constant vector.
\end{theorem}

For the harmonic maps, we have the following result.
\begin{theorem}\label{n-D axi-symmatic}
Let $d\in L^\infty$ be a smooth solution of the (\ref{harmonic maps00}). Assume that
\begin{equation}\label{n-D decay}
 \lim_{|x|\rightarrow \infty} |\nabla d(x)|=0,
\end{equation}
and $d$ is rotationally symmetric, that means $d=d(r)$, where $r=\sqrt{x_1^2+x_2^2+x_3^2}$. Then, $d\in\mathbb{S}^2$ must be a constant vector.

\end{theorem}
The paper is organized as follows. In section 2, we give the proof of Theorem \ref{u=0}. In section 3, we give the proof of Theorem \ref{u=0,decay}. In section 4, we give the proof of Theorem \ref{u=0,decay-1}. In section 5, we give the proof of Theorem \ref{n-D axi-symmatic}.

\section{Proof of Theorem \ref{u=0}}

In this section, we will prove Theorem \ref{u=0}. We use the method of dealing with Steady-state Navier-Stokes equations and the method of dealing with harmonic maps to obtain the Liouville theorem, see \cite{MR1284206} and \cite{MR2431658}.

First, we prove that $u,d$ are smooth under the conditions in Theorem \ref{u=0}. By the equations (\ref{harmonic maps00}) and the condition (\ref{Integration-condition}), we know that $|\Delta d|\in L^{9/4}$. Then, with Sobolev imbedding theorem and Calder\'{o}n-Zygmund theorem, we have
\begin{equation*}
\|\nabla d\|_{L^9}\leq C_1\|\nabla^2 d\|_{L^{\frac{9}{4}}}\leq C_2\|\Delta d\|_{L^{\frac{9}{4}}}.
\end{equation*}
Then, it is easy to see $|\nabla^2 d||\nabla d| \in L^{\frac95}$ by H\"{o}lder inequality. Since $u\in D_0^1$, we have $u\cdot \nabla u\in L^{\frac32}$ by H\"{o}lder inequality. Therefore, by the equation $(\ref{A1})_1$, one has
$$\Delta u \in L_{loc}^{\frac32}.$$
Then, $\nabla u\in L_{loc}^3$ by Sobolev imbedding theorem. Note that $u\in L^6$, we have $u\cdot \nabla u \in L^2_{loc}$ by H\"{o}lder inequality again. Therefore, by the equation $(\ref{A1})_1$, one has
$$\Delta u \in L_{loc}^{\frac95}.$$
Then, $u\in L^\infty_{loc}$ by Sobolev imbedding theorem. Therefore, by the equation $(\ref{A1})_2$, we know $\Delta d\in L^{9/2}_{loc}$. Then, using Sobolev imbedding theorem and Calder\'{o}n-Zygmund theorem again, we have
\begin{equation}\label{3D-Xian-1}
  \|\nabla d\|_{L^\infty_{loc}}\leq C(\|\nabla d\|_{L^{9}},\|\Delta d\|_{L_{loc}^{\frac{9}{2}}}).
\end{equation}

With $u,d \in L^\infty_{loc}$, it is easy to see the solutions of equation (\ref{A1}) are smooth in $\mathbb{R}^3$.

We consider a standard cut-off function $\psi\in C_c^\infty(\mathbb{R})$ such that
$$\psi(y)=\left\{
    \begin{array}{ll}
      1, & \hbox{if $|y|<1$,} \\
      0, & \hbox{if $|y|>2$,}
    \end{array}
  \right.
$$
and $0\leq \psi(y)\leq 1$ for $1<|y|<2$. For each $R$, define
\begin{equation}\label{phi-R}
  \phi_R(x)=\psi(|x|/R),
\end{equation}
$x\in\mathbb{R}^3$.
Note that $\text{div}(\nabla d \odot \nabla d)=\nabla (|\nabla d|^2/2)+\Delta d\cdot \nabla d$. Taking the inner product of $(\ref{A1})_1$ with $u\phi_R$, $(\ref{A1})_2$ with $-\Delta d\phi_R$ in $L^2(\mathbb{R}^3)$. Adding the two resulting integrations together, and integrating by parts, then we have
\begin{equation}\label{u=0_formula_1}
\begin{split}
   & \int_{\mathbb{R}^3}(|\nabla u|^2+|\Delta d|^2)\phi_R dx \\
  =& \int_{\mathbb{R}^3}(\frac{1}{2}|u|^2+\frac{1}{2}|\nabla d|^2 +P)(u\cdot \nabla \phi_R)dx+\int_{\mathbb{R}^n}\frac{1}{2}|u|^2\Delta \phi_R dx\\
   &
  - \int_{\mathbb{R}^3}|\nabla d|^2(d\cdot \Delta d)\phi_R dx.
\end{split}
\end{equation}
Since $|d|=1$, then we have $-|\nabla d|^2=d\cdot \Delta d$. Therefore, by Cauchy inequality and (\ref{u=0_formula_1}), we obtain
\begin{equation}\label{3djieduan}
  \begin{split}
   & \int_{\mathbb{R}^3}|\nabla u|^2\phi_R dx \\
  \leq& \int_{\mathbb{R}^3}(\frac{1}{2}|u|^2+\frac{1}{2}|\nabla d|^2 +P)(u\cdot \nabla \phi_R)dx+\int_{\mathbb{R}^3}\frac{1}{2}|u|^2\Delta \phi_R dx=\sum_{i=1}^4I_i.
\end{split}
\end{equation}
We estimate $I_i$ for $i=1,2,\cdots,4$ one by one. For $I_1$, H\"{o}lder inequality implies
\begin{equation*}
 \begin{split}
   2I_1 & =\int_{\mathbb{R}^3}|u|^2(u\cdot \nabla \phi_R)dx \leq \int_{R\leq |x|\leq 2R}|u|^3|\nabla \phi_R|dx\\
     & \leq \frac{1}{R}||\nabla \psi||_{L^\infty}\int_{R\leq |x|\leq 2R}|u|^3dx\\
     &\leq \frac{1}{R}||\nabla \psi||_{L^\infty}\left (\int_{R\leq |x|\leq 2R}|u|^{\frac{9}{2}}dx\right)^ {\frac{2}{3}}\left (\int_{R\leq |x|\leq 2R}dx\right)^ {\frac{1}{3}}\\
     &\leq C||\nabla \psi||_{L^\infty}\left (\int_{R\leq |x|\leq 2R}|u|^{\frac{9}{2}}dx\right)^ {\frac{2}{3}}\rightarrow 0 \quad (\text{as} \, R\rightarrow \infty)
 \end{split}
\end{equation*}
With the estimate of $I_1$, we have
\begin{equation*}
 \begin{split}
   2I_2 & =\int_{\mathbb{R}^3}|\nabla d|^2(u\cdot \nabla \phi_R)dx \leq \int_{R\leq |x|\leq 2R}|\nabla d|^2|u||\nabla \phi_R|dx\\
     & \leq \frac{1}{R}||\nabla \psi||_{L^\infty}\int_{R\leq |x|\leq 2R}|\nabla d|^2|u|dx\\
     &\leq \left (\frac{1}{R}||\nabla \psi||_{L^\infty}\int_{R\leq |x|\leq 2R}|u|^3dx\right)^ {\frac{1}{3}}\left (\frac{1}{R}||\nabla \phi||_{L^\infty}\int_{R\leq |x|\leq 2R}|\nabla d|^3dx\right)^ {\frac{2}{3}}\\
     &\rightarrow 0 \quad (\text{as} \, R\rightarrow \infty)
 \end{split}
\end{equation*}
By the Calder\'{o}n-Zygmund theorem,  we have $P\in L^{9/4}_{x,t}(\mathbb{R}^3)$, then
\begin{equation*}
 \begin{split}
   I_3 & =\int_{\mathbb{R}^3}P(u\cdot \nabla \phi_R)dx \leq \frac{1}{R}||\nabla \phi||_{L^\infty}\int_{R\leq |x|\leq 2R}|P||u|dx\\
     &\leq \frac{C}{R}||\nabla \phi||_{L^\infty}\left (\int_{R\leq |x|\leq 2R}|u|^{\frac95}dx\right)^ {\frac{5}{9}}\left (\int_{R\leq |x|\leq 2R}|P|^{\frac{9}{4}}dx\right)^ {\frac{4}{9}}\\
     &\leq C||\nabla \phi||_{L^\infty}\left (\int_{R\leq |x|\leq 2R}|u|^{\frac92}dx\right)^ {\frac{2}{9}}\left (\int_{R\leq |x|\leq 2R}|P|^{\frac{9}{4}}dx\right)^ {\frac{4}{9}}\\
     &\rightarrow 0 \quad (\text{as} \, R\rightarrow \infty)
 \end{split}
\end{equation*}
Finally, for $I_4$, we have
\begin{equation*}
 \begin{split}
   2I_4 & =\int_{\mathbb{R}^3}|u|^2\Delta\phi_Rdx \leq \frac{1}{R^2}||\Delta \phi||_{L^\infty}\int_{R\leq |x|\leq 2R}|u|^2dx\\
     &\leq \frac{1}{R^2}||\Delta \psi||_{L^\infty}\left (\int_{R\leq |x|\leq 2R}|u|^{\frac{9}{2}}dx\right)^ {\frac{4}{9}}\left (\int_{R\leq |x|\leq 2R}dx\right)^ {\frac{5}{9}}\\
     &\leq CR^{-\frac{1}{3}}\left (\int_{R\leq |x|\leq 2R}|u|^{\frac{9}{2}}dx\right)^ {\frac{4}{9}}\rightarrow 0 \quad (\text{as} \, R\rightarrow \infty)
 \end{split}
\end{equation*}
Therefore, with the above estimates, we get
\begin{equation*}
  \int_{\mathbb{R}^3}|\nabla u|^2dx =0.
\end{equation*}
Consequently, $u$ is a constant vector, and is zero due to $u\in L^{\frac92}(\mathbb{R}^3) $.\\
With $u=0$ and the equation $(\ref{A1})_3$, we have the equations (\ref{harmonic maps00}).

Denote $d_{i,j}=\frac{\partial d_i}{\partial x_j}$. Taking the inner product of $(\ref{harmonic maps00})$ with $x\cdot\nabla d\phi_R$ in $L^2(\mathbb{R}^3)$, we have
\begin{equation}\label{d-L1}
  \int_{\mathbb{R}^3} \Delta d_{i}x_j d_{i,j}\phi_R=\int_{\mathbb{R}^3} |\nabla d|^2 x_jd_i d_{i,j}\phi_R=0,
\end{equation}
where we use the fact that $\partial_{x_j} |d|^2 =0$ for $j=1,2,3$. Then, integrating by parts, one has
\begin{eqnarray*}
    && \int_{\mathbb{R}^3} \Delta d_{i}x_j d_{i,j}\phi_R \\
   &=& -\int_{\mathbb{R}^3} (d_{i,k}\delta_{j,k} d_{i,j}\phi_R +d_{i,k}x_jd_{i,jk}\phi_R+d_{i,k}x_jd_{i,j}\partial_{x_k}\phi_{R})\\
   &=& -\int_{\mathbb{R}^3} (|\nabla d|^2\phi_R +d_{i,k}x_jd_{i,j}\partial_{x_k}\phi_{R}) +\int_{\mathbb{R}^3}(\frac32|\nabla d|^2\phi_R +\frac12 |\nabla d|^2 x_j \partial_{x_j} \phi_R)  \\
   &=&\frac12\int_{\mathbb{R}^3} |\nabla d|^2\phi_R +\int_{\mathbb{R}^3}(\frac12 |\nabla d|^2 x_j \partial_{x_j}\nabla \phi_R-d_{i,k}x_jd_{i,j}\partial_{x_k}\phi_{R}),
\end{eqnarray*}
where
$$\delta _{jk}=\left\{
                 \begin{array}{ll}
                   1, & \hbox{$x_j=x_k$;} \\
                   0, & \hbox{$x_j\neq x_k$.}
                 \end{array}
               \right.
$$
Therefore, by the definition of $\phi_R$, one has
$$\int_{\mathbb{R}^3} |\nabla d|^2\phi_R\leq C\int_{R\leq |x|\leq 2R} |\nabla d|^2\rightarrow 0 \quad (\text{as} \, R\rightarrow \infty).$$

\section{Proof of Theorem \ref{u=0,decay}}

By the above section, we known that $u, d$ are smooth. By the equation (\ref{A1}), it is easy to see
$$\Delta P= -\text{div div}\, (u\otimes u + \nabla d \odot \nabla d).$$
Then we have
\begin{equation}\label{p-formula}
  P(x)=-\frac13 (|u(x)|^2+ |\nabla d(x)|^2) +P.V.\int_{\mathbb{R}^3} \partial_{y_i}\partial_{y_j}\Gamma(x-y)F_{ij}(y)dy,
\end{equation}
where $\Gamma =\frac{1}{4\pi|x|} $ and $F_{ij}=u_iu_j+\partial_{x_i}d\cdot \partial_{x_j} d$. The proof of equation (\ref{p-formula}) is similar to the Navier-Stokes case(see \cite{na} and \cite{MR2925135}).

Let $x\in (B^\prime_{2R}\setminus B^\prime_R)\times (L_{2R}\setminus L_R)$, $\psi_R=\phi_{9R}-\phi_{R/4}$, $\sigma_R(x_1,x_2,x_3)=\psi_R(x_1,x_2)\psi_R(x_3)$(see the equation (\ref{phi-R}) for the definition of $\phi_R$ in section 2), where $B^\prime_R=\{\sqrt{x_1^2+x_2^2}\leq R\}$ and $L_R= \{|y_3|\leq R \}$. Then
\begin{eqnarray}\label{p-f-0}
  P(x) &=& -\frac13 (|u(x)|^2+ |\nabla d(x)|^2) +P.V.\int_{\mathbb{R}^3} \partial_{y_i}\partial_{y_j}\Gamma(x-y)F_{ij}(y)\sigma_R(y)dy \\
       && +\int_{\mathbb{R}^3} \partial_{y_i}\partial_{y_j}\Gamma(x-y)F_{ij}(y)(1-\sigma_R(y)))dy .\nonumber
\end{eqnarray}
Note that $\sigma(x)=1$ in $(B^\prime_{9R}\setminus B^\prime_{R/2})\times(L_{9R}\setminus L_{R/2})$ and supp$\,\sigma \subset (B^\prime_{18R}\setminus B^\prime_{R/4})\times (L_{18R}\setminus L_{R/4})$. Therefore,
\begin{eqnarray}\label{p-f-1}
  |P(x)| &\leq& \frac13 (|u(x)|^2+ |\nabla d(x)|^2) +\mid P.V.\int_{\mathbb{R}^3} \partial_{y_i}\partial_{y_j}\Gamma(x-y)F_{ij}(y)\sigma_R(y)dy\mid \\
       && +C\int_{D_1} \frac{1}{|x-y|^3}|F_{ij}(y)|dy + C\int_{D_2} \frac{1}{|x-y|^3}|F_{ij}(y)|dy\nonumber\\
       &&+C\int_{D_3} \frac{1}{|x-y|^3}|F_{ij}(y)|dy+C\int_{D_4} \frac{1}{|x-y|^3}|F_{ij}(y)|dy \nonumber\\
       &:=& \sum_{i=1}^6 J_i(x), \nonumber
\end{eqnarray}
where
$$D_1=B^\prime_{9R}\times L_{R/2}\cap B^\prime_{R/2}\times(L_{9R}\setminus L_{R/2}),$$
$$D_2=B^\prime_{9R}\times(\mathbb{R}\setminus L_{9R}),\quad D_3=(\mathbb{R}^3\setminus B_{9R})\times L_{9R},$$
$$D_4=(\mathbb{R}^3\setminus B_{9R})\times (\mathbb{R}\setminus L_{9R}).$$

Then we will estimate $J_1,\cdots,J_6$ step by step. By the condition (\ref{decay}) in Theorem \ref{u=0,decay}, one has
\begin{equation}\label{j-1}
  \int_{R\leq |x_3|\leq 2R}\int_{R\leq r\prime\leq 2R}|J_1(x)|^{\frac32}dx\leq C R^{3-3\alpha},
\end{equation}
where $r\prime=\sqrt{x_1^2+x_2^2}$ and $C$ is independent on $R$. And by Calder\'{o}n-Zygmund theorem, we have
\begin{equation}\label{j-2}
   \int_{R\leq |x_3|\leq 2R}\int_{R\leq r\prime\leq 2R} |J_2(x)|^{\frac32}dx \leq C\sum_{i,j}\int_{\mathbb{R}^3}|F_{ij}(y)\sigma(y)|^{\frac32}dy\leq CR^{3-3\alpha},
\end{equation}
where $C$ is independent on $R$.\\
 Now, we give the decay rate of $J_3,J_4,J_5,J_6$ in the equation (\ref{p-f-1}). By the condition (\ref{decay}) in Theorem \ref{u=0,decay}, one can see that $|F_{ij}(y_1,y_2,y_3)|\leq C/(1+r\prime)^{2\alpha}$, where $\alpha>\frac43$. Let $R>2$, then

\begin{eqnarray}\label{j-3}
  |J_3| &=&C\int_{D_1} \frac{1}{|x-y|^3}|F_{ij}(y)|dy \\
       &\leq & \frac{2}{R^3}\int_{D_1} |F_{ij}(y)|dy \nonumber\\
       &\leq& \frac{C}{R^3}\int_{D_1} \frac{1}{(1+r\prime)^{2\alpha}}dy \nonumber\\
       &\leq& \frac{C}{R^2}\int_{r\prime\leq 9R} \frac{1}{(1+r\prime)^{2\alpha}}dy_1dy_2  \nonumber\\
       &\leq& CR^{-2\gamma},\nonumber
\end{eqnarray}
where $\gamma =\min\{1,\alpha\}$ and $C$ is independent on $R$. Obviously, $\gamma>\frac23$. It is easy to see that $|y-x|\geq |y|-|x|\geq \frac12 (\sqrt{y_1^2+y_2^2}+|y_3|)-|x|$. For $J_4$, we have
\begin{eqnarray*}
  J_4 &=& C\int_{D_2} \frac{1}{|x-y|^3}|F_{ij}(y)|dy\\
    &\leq& C\int_{r\prime \leq 9R} \int_{9R}^\infty  \frac{1}{\mid|y_3|+r\prime -2|x|\mid^3}\frac{1}{(1+r\prime)^{2\alpha}}dy_3dy_1dy_2\nonumber \\
   &\leq& C \int_{r\prime \leq 9R} \frac{1}{\mid 9 R +r\prime -2|x|\mid^2}\frac{1}{(1+r\prime)^{2\alpha}}dy_1dy_2\nonumber \\
   &\leq& C \frac{1}{R^2}\int_{r\prime \leq 9R} \frac{1}{(1+r\prime)^{2\alpha}}dy_1dy_2\nonumber \\
   &\leq& C \frac{1}{R^2}(1+R^{2-2\alpha})\nonumber\\
   &\leq& CR^{-2\gamma},\nonumber
\end{eqnarray*}
where $\gamma$ is same as the equation (\ref{j-3}) and $C$ is independent on $R$. For $J_5$, we have
\begin{eqnarray*}
  J_5 &=&C\int_{D_3} \frac{1}{|x-y|^3}|F_{ij}(y)|dy\\
  &\leq& C\int_{r\prime\geq 9 R} \int_{0}^{9R}  \frac{1}{\mid|y_3|+r\prime -2|x|\mid^3}\frac{1}{(1+r\prime)^{2\alpha}}dy_3dy_1dy_2\nonumber \\
  &\leq&   C\int_{r\prime\geq 9 R} \left(\frac{1}{\mid r\prime -2|x|\mid^2} +\frac{1}{\mid 9 R+r\prime -2|x|\mid^2}\right)\frac{1}{(1+r\prime)^{2\alpha}}dy_1dy_2\nonumber\\
   &\leq& C \int_{r\prime\geq 9 R} \left(\frac{1}{\mid r\prime -2|x|\mid^{2+2\alpha}} +\frac{1}{\mid 1+r\prime \mid^{2+2\alpha}}\right)dy_1dy_2\nonumber\\
   &\leq& C \left(\frac{1}{\mid 9R -2|x|\mid^{2\alpha}} +\frac{1}{\mid 1+9 R \mid^{2\alpha}}\right)\nonumber\\
   &\leq& CR^{-2\alpha},\nonumber
\end{eqnarray*}
where $C$ is independent on $R$.
For $J_6$, we have
\begin{eqnarray*}
  J_6 &=&C\int_{D_4} \frac{1}{|x-y|^3}|F_{ij}(y)|dy\\
   &\leq& C\int_{r\prime >9R} \int_{9 R}^\infty  \frac{1}{\mid|y_3|+r\prime -2|x|\mid^3}\frac{1}{(1+r\prime)^{2\alpha}}dy_3dy_1dy_2\nonumber\\
   &\leq& C \int_{r\prime > 9R} \frac{1}{\mid 9 R +r\prime -2|x|\mid^2}\frac{1}{(1+r\prime)^{2\alpha}}dy_1dy_2 \nonumber \\
   &\leq& C \int_{r\prime > 9R} \frac{1}{(1+r\prime)^{2+2\alpha}}dy_1dy_2\nonumber\\
  &\leq & CR^{-2\alpha},\nonumber
\end{eqnarray*}
where $C$ is independent on $R$.
 Therefore,
\begin{equation}\label{j-3-6}
  |J_3|+|J_4|+|J_5|+|J_6| \leq CR^{-2\gamma},
\end{equation}
where $\gamma$ is same as the equation (\ref{j-3}) and $C$ is independent on $R$. \\
Combining (\ref{p-f-1}), (\ref{j-1}), (\ref{j-2}) and (\ref{j-3-6}), we have
\begin{eqnarray}
  \int_{R\leq |y_3|\leq 2R}\int_{R\leq r\prime\leq 2R} |P(x)|^{\frac32} dx &\leq& \sum_{i=1}^6 \int_{R\leq |y_3|\leq 2R}\int_{R\leq r\prime\leq 2R}|J_i(x)|^{\frac32}dx \\
  &\leq& C R^{3-3\gamma},\nonumber
\end{eqnarray}
where $\gamma$ is same as the equation (\ref{j-3}) and $C$ is independent on $R$. \\
With (\ref{3djieduan}), we have
\begin{eqnarray*}
     && \int_{\mathbb{R}^3}|\nabla u|^2\eta_R dx \\
  &\leq& \int_{\mathbb{R}^3}(\frac{1}{2}|u|^2+\frac{1}{2}|\nabla d|^2 +P)(u\cdot \nabla \eta_R)dx+\int_{\mathbb{R}^n}\frac{1}{2}|u|^2\Delta \eta_R dx\\
  &\leq &\frac{C}{R}\int_{D}(|u|^3+|u||\nabla d|^2 +|Pu|)dx +\frac{C}{R^2}\int_{D}|u|^2dx\\
  &\leq&\frac{C}{R}\left(\int_{D}|P|^{\frac32}\right)^{\frac23}\left(\int_{D}|u|^3\right)^{\frac13}+\frac{C}{R}R^{3-3\alpha}+\frac{C}{R^2}R^{3-2\alpha}\\
  &\leq&\frac{C}{R}R^{2-2\gamma}R^{1-\alpha} +CR^{2-3\alpha}+CR^{1-2\alpha}\\
  &\leq&CR^{2-2\gamma-\alpha} +CR^{2-3\alpha}+CR^{1-2\alpha},
\end{eqnarray*}
where $\eta_R(x_1,x_2,x_3)=\phi_R(x_1,x_2)\phi_R(x_3)$ and $D=(B^\prime_{2R}\setminus B^\prime_R)\times(L_{2R}\setminus L_{R})$. Since $\gamma>\frac23$ and $\alpha>\frac23$, it is easy to see that
$$2-2\gamma-\alpha<0,\quad 2-3\alpha<0,\quad 1-2\alpha<0.$$
Let $R\rightarrow \infty$, we have
$$\int_{\mathbb{R}^3} |\nabla u|^2=0.$$
Then, $u\equiv 0$ by $u\in L^6$. Therefore, $d$ satisfies the harmonic maps with $d\in \mathring{H}$, then $d$ is constant vector by the above section.

\section{Proof of Theorem \ref{u=0,decay-1}}
In this section, the notation is same as section 3. To prove the theorem \ref{u=0,decay-1}, we also estimate the pressure $P$ in the first step.
\begin{eqnarray}\label{p-f-2}
  |P(x)| &\leq& \frac13 (|u(x)|^2+ |\nabla d(x)|^2) +\mid P.V.\int_{\mathbb{R}^3} \partial_{y_i}\partial_{y_j}\Gamma(x-y)F_{ij}(y)\sigma_R(y)dy\mid \\
       && +C\int_{D_1} \frac{1}{|x-y|^3}|F_{ij}(y)|dy + C\int_{D_2} \frac{1}{|x-y|^3}|F_{ij}(y)|dy\nonumber\\
       &&+C\int_{D_3} \frac{1}{|x-y|^3}|F_{ij}(y)|dy+C\int_{D_4} \frac{1}{|x-y|^3}|F_{ij}(y)|dy \nonumber\\
       &:=& \sum_{i=1}^6 K_i(x), \nonumber
\end{eqnarray}
where
$$D_1=B^\prime_{9R}\times L_{R/2}\cap B^\prime_{R/2}\times(L_{9R}\setminus L_{R/2}),$$
$$D_2=B^\prime_{9R}\times(\mathbb{R}\setminus L_{9R}),\quad D_3=(\mathbb{R}^3\setminus B_{9R})\times L_{9R},$$
$$D_4=(\mathbb{R}^2\setminus B_{9R})\times (\mathbb{R}\setminus L_{9R}).$$
Then we will estimate $K_1,\cdots,K_6$ step by step. By the condition (\ref{decay-1}) in Theorem \ref{u=0,decay-1}, one has
\begin{equation}\label{k-1}
  \int_{R\leq |y_3|\leq 2R}\int_{R\leq r\prime\leq 2R}|K_1(x)|^{\frac32}dx\leq C R^{3-3\beta},
\end{equation}
where $r\prime=\sqrt{x_1^2+x_2^2}$ and $C$ is independent on $R$. And by Calder\'{o}n-Zygmund theorem, we have
\begin{equation}\label{k-2}
   \int_{R\leq |y_3|\leq 2R}\int_{R\leq r\prime\leq 2R} |K_2(x)|^{\frac32}dx \leq C\sum_{i,j}\int_{\mathbb{R}^3}|F_{ij}(y)\sigma(y)|^{\frac32}dy\leq CR^{3-3\beta},
\end{equation}
where $C$ is independent on $R$.\\
 Now, we give the decay rate of $K_3,K_4,K_5,K_6$ in the equation (\ref{p-f-2}). By the condition (\ref{decay-1}) in Theorem \ref{u=0,decay-1}, one can see that $|F_{ij}(y_1,y_2,y_3)|\leq C/(1+|y_3|)^{2\beta}$, where $\beta>1$. Let $R>2$, then

\begin{eqnarray}\label{k-3}
  |K_3| &=&C\int_{D_1} \frac{1}{|x-y|^3}|F_{ij}(y)|dy \\
       &\leq & \frac{2}{R^3}\int_{D_1} |F_{ij}(y)|dy \nonumber\\
       &\leq& \frac{C}{R^3}\int_{D_1} \frac{1}{(1+|y_3|)^{2\beta}}dy \nonumber\\
       &\leq& \frac{C}{R}\int_{0}^{9R} \frac{1}{(1+|y_3|)^{2\beta}}dy_3  \nonumber\\
       &\leq& CR^{-1},\nonumber
\end{eqnarray}
where $C$ is independent on $R$. For $K_4$, we have
\begin{eqnarray*}
  K_4 &=& C\int_{D_2} \frac{1}{|x-y|^3}|F_{ij}(y)|dy\\
    &\leq& C\int_{r\prime \leq 9R} \int_{9R}^\infty  \frac{1}{\mid|y_3|+r\prime -2|x|\mid^3}\frac{1}{(1+|y_3|)^{2\beta}}dy_3dy_1dy_2\nonumber \\
   &\leq& \frac{C}{R^{2\beta-1}} \int_{r\prime \leq 9R} \frac{1}{\mid 9 R +r\prime -2|x|\mid^2}dy_1dy_2\nonumber \\
   &\leq& \frac{C}{R^{2\beta-1}}(1+ln R)\nonumber\\
   &\leq& CR^{-1},\nonumber
\end{eqnarray*}
where $C$ is independent on $R$. For $K_5$, we have
\begin{eqnarray*}
  K_5 &=&C\int_{D_3} \frac{1}{|x-y|^3}|F_{ij}(y)|dy\\
  &\leq& C\int_{r\prime\geq 9 R} \int_{0}^{9R}  \frac{1}{\mid|y_3|+r\prime -2|x|\mid^3}\frac{1}{(1+|y_3|)^{2\beta}}dy_3dy_1dy_2\nonumber \\
   &\leq& C\int_{r\prime\geq 9 R} \int_{0}^{9R}  \frac{1}{\mid r\prime -2|x|\mid^3}\frac{1}{(1+|y_3|)^{2\beta}}dy_3dy_1dy_2\nonumber \\
  &\leq&   C\int_{r\prime\geq 9 R} \frac{1}{\mid r\prime -2|x|\mid^3} dy_1dy_2\nonumber\\
   &\leq& C \frac{1}{\mid 9R -2|x|\mid} \nonumber\\
   &\leq& CR^{-1},\nonumber
\end{eqnarray*}
where $C$ is independent on $R$.
For $K_6$, we have
\begin{eqnarray*}
  K_6 &=&C\int_{D_4} \frac{1}{|x-y|^3}|F_{ij}(y)|dy\\
   &\leq& C\int_{r\prime >9R} \int_{9 R}^\infty  \frac{1}{\mid|y_3|+r\prime -2|x|\mid^3}\frac{1}{(1+|y_3|)^{2\beta}}dy_3dy_1dy_2\nonumber\\
   &\leq& C \int_{ 9R}^\infty \frac{1}{\mid 9 R +|y_3| -2|x|\mid}\frac{1}{(1+|y_3|)^{2\beta}}dy_3 \nonumber \\
   &\leq& C \int_{ 9R}^\infty \frac{1}{(1+|y_3|)^{1+2\beta}}dy_1dy_2\nonumber\\
  &\leq & CR^{-2\beta},\nonumber
\end{eqnarray*}
where $C$ is independent on $R$.
 Therefore,
\begin{equation}\label{k-3-6}
  |K_3|+|K_4|+|K_5|+|K_6| \leq CR^{-1},
\end{equation}
where and $C$ is independent on $R$. \\
Combining (\ref{p-f-2}), (\ref{k-1}), (\ref{k-2}) and (\ref{k-3-6}), we have
\begin{eqnarray}
  \int_{R\leq |y_3|\leq 2R}\int_{R\leq r\prime\leq 2R} |P(x)|^{\frac32} dx &\leq& \sum_{i=1}^6 \int_{R\leq |y_3|\leq 2R}\int_{R\leq r\prime\leq 2R}|K_i(x)|^{\frac32}dx \\
  &\leq& C R^{\frac32},\nonumber
\end{eqnarray}
where $C$ is independent on $R$. \\
With (\ref{3djieduan}), we have
\begin{eqnarray*}
     && \int_{\mathbb{R}^3}|\nabla u|^2\eta_R dx \\
  &\leq& \int_{\mathbb{R}^3}(\frac{1}{2}|u|^2+\frac{1}{2}|\nabla d|^2 +P)(u\cdot \nabla \eta_R)dx+\int_{\mathbb{R}^n}\frac{1}{2}|u|^2\Delta \eta_R dx\\
  &\leq &\frac{C}{R}\int_{D}(|u|^3+|u||\nabla d|^2 +|Pu|)dx +\frac{C}{R^2}\int_{D}|u|^2dx\\
  &\leq&\frac{C}{R}\left(\int_{D}|P|^{\frac32}\right)^{\frac23}\left(\int_{D}|u|^3\right)^{\frac13}+\frac{C}{R}R^{3-3\beta}+\frac{C}{R^2}R^{3-2\beta}\\
  &\leq&\frac{C}{R}R^{1}R^{1-\beta} +CR^{2-3\beta}+CR^{1-2\beta}\\
  &\leq&CR^{1-\beta} +CR^{2-3\beta}+CR^{1-2\beta},
\end{eqnarray*}
where $\eta_R(x_1,x_2,x_3)=\phi_R(x_1,x_2)\phi_R(x_3)$ and $D=(B\prime_{2R}\setminus B^\prime_R)\times(L_{2R}\setminus L_{R})$. Since $\beta>1$, it is easy to see that
$$1-\beta<0,\quad 2-3\beta<0,\quad 1-2\beta<0.$$
Let $R\rightarrow \infty$, we have
$$\int_{\mathbb{R}^3} |\nabla u|^2=0.$$

\section{The Liouville Theorem of harmonic maps (\ref{harmonic maps00})}

\begin{proof}[Proof of Theorem \ref{n-D axi-symmatic}]
With the equations (\ref{harmonic maps00}), $d_{k,i}$ satisfies
\begin{equation}\label{n-D--d-k,i}
-\Delta d_{k,i}=|\nabla d|^2 \cdot d_{k,i}+2d_k\cdot \sum_{l,j=1}^nd_{l,j}\cdot d_{l,ji},
\end{equation}
for $k=1,2,3,\,i=1,2,3$.
Multiplying both sides of equation (\ref{n-D--d-k,i}) by $d_{k,i}$ and taking sum of them, we have
\begin{equation*}
  -\Delta \frac{|\nabla d|^2}{2}=-\sum_{k=1}^3\sum_{i,j=1}^3 d_{k,ij}+|\nabla d|^4 + 2\sum_{i=1}^3\left(\sum_{k=1}^3 d_{k,i}\cdot d_k\right)\cdot \left (\sum_{l,j=1}^3 d_{l,j}\cdot d_{l,ji}\right).
\end{equation*}

Since $|d|^2=1$, it is easy to see that $\sum_{k=1}^3 d_k\cdot d_{k,i}=0$ for $i=1,2,\cdots,n$. Therefore, we have
\begin{equation}\label{n-D nabla}
  -\Delta \frac{|\nabla d|^2}{2}=-\sum_{k=1}^3\sum_{i,j=1}^3 d_{k,ij}^2+|\nabla d|^4.
\end{equation}

By direct calculation, for $d=d(r)$, we have
\begin{equation}\label{nabla  mo}
 |\nabla d|^2=\sum_{k=1}^3 d_{k,r}^2,
\end{equation}

\begin{equation}\label{nabla2 mo}
\begin{split}
  \sum_{k=1}^3\sum_{i,j=1}^3 d_{k,ij}^2   &=  \sum_{k=1}^3\sum_{i,j=1}^3 \left (d_{k,rr}\cdot \frac{x_i\cdot x_j}{r^2}+d_{k,r}\cdot\frac{\delta_{ij}}{r}-d_{k,r}\cdot\frac{x_i\cdot x_j}{r^3}\right )^2\\
    & =\sum_{k=1}^3\left( d_{k,rr}^2+d_{k,r}^2\frac{n-1}{r^2}\right )\\
  & =\frac{2}{r^2}\cdot |\nabla d|^2 +\sum_{k=1}^3 d_{k,rr}^2.
\end{split}
\end{equation}

Due to $d=d(r)$ and the equation (\ref{harmonic maps00}), we have
\begin{equation*}
 -d_{k,rr}=d_{k,r}\cdot \frac{2}{r}+d_k\cdot|\nabla d|^2.
\end{equation*}
Then, with the equation (\ref{nabla  mo}), we obtain
\begin{equation}\label{d-k,rr}
 \begin{split}
    \sum_{k=1}^3 d_{k,rr}^2 & = \sum_{k=1}^3 \left (d_{k,r}\cdot \frac{2}{r}+d_k\cdot |\nabla d|^2 \right)^2\\
     & =\sum_{k=1}^3\left (d_{k,r}^2\cdot \frac{(n-1)^2}{r^2}+ 2d_k\cdot d_{k,r}\cdot |\nabla d|^2\cdot\frac{n-1}{r}  + d_k^2\cdot |\nabla d|^4 \right)\\
 &=\frac{(n-1)^2}{r^2}\cdot |\nabla d|^2+ |\nabla d|^4.
 \end{split}
\end{equation}

With the equations (\ref{n-D nabla}), (\ref{nabla2 mo}) and (\ref{d-k,rr}), we have
\begin{equation*}
  -\Delta \frac{|\nabla d|^2}{2}=-\frac{6}{r^2}|\nabla d|^2\leq 0 \quad \text{in } \mathbb{R}^3.
\end{equation*}
Then, the result follows the standard maximum principle for elliptic equations(see \cite{GT01}).

\end{proof}

\section*{Acknowledgments}  The authors would like to thank Prof. J.X. Hong and D.r. Tao Huang for helpful comments. Yihang Hao is partial supported by NSFC no.11526068. Xiangao Liu is partial supported by NSFC no.11631011.


\begin{thebibliography}{abcde}
\bibitem{MR1284206} G. P. Galdi, An introduction to the mathematical theory of the {N}avier-{S}tokes equations. {V}ol. {II}, Springer-Verlag, New York, 1994, 39, xii+323.

\bibitem{Leslie1968} F.M. Leslie, Some constitutive equations for liquid crystals, Archive for Rational Mechanics and Analysis, 1968, 28(4), 265--283.

\bibitem{MR0137403} J.L. Ericksen, Hydrostatic theory of liquid crystals, Arch. Rational Mech. Anal., 1962, 9, 371--378.

\bibitem{de1974} P.G. de Gennes, The physics of liquid crystals, 1974, Oxford University Press, Oxford.

\bibitem{MR1003435} F.H. Lin, Nonlinear theory of defects in nematic liquid crystals; phase transition and flow phenomena, Comm. Pure Appl. Math., 1989, 42(6), 789--814.

\bibitem{MR2646822} F.H. Lin, J.Y. Lin, C.Y. Wang, Liquid crystal flows in two dimensions, Arch. Ration. Mech. Anal., 2010, 197(1), 297--336.

\bibitem{MR2745211} F.H. Lin, C.Y. Wang, On the uniqueness of heat flow of harmonic maps and hydrodynamic flow of nematic liquid crystals, Chin. Ann. Math. Ser. B, 2010, 31(6), 921--938.

\bibitem{MR2745194} M.C. Hong, Global existence of solutions of the simplified {E}ricksen-{L}eslie system in dimension two, Calc. Var. Partial Differential Equations, 2011, 40(1-2), 15--36.

\bibitem{MR2853534} X. Xu, Z.F. Zhang, Global regularity and uniqueness of weak solution for the 2-{D} liquid crystal flows, J. Differential Equations, 2012, 252(2), 1169--1181.

\bibitem{MR3518239} F.H. Lin, C.Y. Wang, Global existence of weak solutions of the nematic liquid crystal flow in dimension three, Comm. Pure Appl. Math., 2016, 69(8), 1532--1571.

\bibitem{MR3548261} D. Chae, J. Wolf, On {L}iouville type theorems for the steady {N}avier-{S}tokes equations in {$\Bbb{R}^3$}, J. Differential Equations, 2016, 261(10), 5541--5560.

\bibitem{MR3538409} G. Seregin, Liouville type theorem for stationary {N}avier-{S}tokes equations, Nonlinearity, 2016, 29(8), 2191--2195.

\bibitem{MR3571910} H. Kozono, Y. Terasawa, Y. Wakasugi, A remark on {L}iouville-type theorems for the stationary {N}avier-{S}tokes equations in three space dimensions, J. Funct. Anal., 2017, 272(2), 804--818.

\bibitem{MR3162482} D. Chae, Dongho, Liouville-type theorems for the forced {E}uler equations and the {N}avier-{S}tokes equations, Comm. Math. Phys., 2014, 326(1), 37--48.

\bibitem{MR3061045} D. Chae, T. Yoneda, On the {L}iouville theorem for the stationary {N}avier-{S}tokes equations in a critical space, J. Math. Anal. Appl., 2013, 405(2), 706--710.

\bibitem{CPZ} C. Bryan, X. Pan, Q. S. Zhang, Decay and vanishing of some axially symmetric {D}-solutions of the Navier-Stokes equations, arXiv:1801.07420

\bibitem{SG}  G. Seregin, Remarks on Liouville Type Theorems for Steady-State Navier-Stokes Equations, arXiv:1703.10822

\bibitem{SW} G. Seregin, W. Wang, Sufficient conditions on Liouville type theorems for the 3D steady Navier-Stokes equations, arXiv:1805.02227


\bibitem{MR3014091} M. Fuchs, X. Zhong, A note on a {L}iouville type result of {G}ilbarg and {W}einberger for the stationary {N}avier-{S}tokes equations in {$2D$}, J. Math. Sci. (N.Y.), 2011, 178(6), 695--703.

\bibitem{MR3543547} D. Chae, S. Weng, Liouville type theorems for the steady axially symmetric {N}avier-{S}tokes and magnetohydrodynamic equations, Discrete Contin. Dyn. Syst., 2016, 36(10), 5267--5285.

\bibitem{MR3345358} M. Korobkov, K. Pileckas, R. Russo, The {L}iouville theorem for the steady-state {N}avier-{S}tokes problem for axially symmetric 3{D} solutions in absence of swirl, J. Math. Fluid Mech., 2015, 17(2), 287--293.

\bibitem{dong} W.D. Wang, Remarks on {L}iouville type theorems for the 3{D} steady axially symmetric {N}avier-{S}tokes equations, J. Differential Equations, 2019, 266(10), 6507--6524.

\bibitem{na} N. Zhao, A Liouville type theorem for axially symmetric {D}-solutions to steady Navier-Stokes Equations, arXiv:1805.03845

\bibitem{MR0431040} S.T. Yau, Harmonic functions on complete {R}iemannian manifolds, Comm. Pure Appl. Math., 1975, 28, 201--228.

\bibitem{MR573431} S. Y. Cheng,  Liouville theorem for harmonic maps, Geometry of the {L}aplace operator ({P}roc. {S}ympos. {P}ure {M}ath., {U}niv. {H}awaii, {H}onolulu, {H}awaii, 1979), Amer. Math. Soc., Providence, R.I., 1980, 147--151.

\bibitem{MR647905} H.I. Choi, On the {L}iouville theorem for harmonic maps, Proc. Amer. Math. Soc., 1982, 85(1), 91--94.

\bibitem{MR1362965} L. F. Tam, Liouville properties of harmonic maps, Math. Res. Lett., 1995, 2(6), 719--735.

\bibitem{MR1333944} Y. Shen, A {L}iouville theorem for harmonic maps, Amer. J. Math., 1995, 117(3), 773--785.

\bibitem{MR762354} R. Schoen, K. Uhlenbeck, Regularity of minimizing harmonic maps into the sphere, Invent. Math., 1984, 78(1), 89--100.


\bibitem{MR895408} Y. L. Xin, Liouville type theorems and regularity of harmonic maps, Differential geometry and differential equations ({S}hanghai, 1985), Springer, Berlin, 1987, 1255, 198--208.

\bibitem{MR1156381} Z. R. Jin, Liouville theorems for harmonic maps, Invent. Math., 1992, 108(1), 1--10.

\bibitem{MR1695783} M. Rigoli, A. G. Setti, Liouville-type theorems for rotationally symmetric harmonic maps, Dynam. Systems Appl., 1999, 8(2), 243--263.

\bibitem{MR3290380} Z. Sinaei, Riemannian polyhedra and {L}iouville-type theorems for harmonic maps, Anal. Geom. Metr. Spaces, 2014, 2, 294--318.


\bibitem{MR2431658} F.H. Lin, C.Y. Wang, The analysis of harmonic maps and their heat flows, World Scientific Publishing Co. Pte. Ltd., Hackensack, NJ, 2008, xii+267.

\bibitem{MR2925135} G. Seregin, A certain necessary condition of potential blow up for {N}avier-{S}tokes equations, Comm. Math. Phys., 2012, 312(3), 833--845.

\bibitem{GT01} D. Gilbarg, N.S. Trudinger, Elliptic partial differential equations of second order, Springer-Verlag, Berlin, Reprint of the 1998 edition, 2001, xiv+517.



\end{thebibliography}
\end{document}